\documentclass{birkjour}
 \newtheorem{thm}{Theorem}[section]
 
 \newtheorem{lem}[thm]{Lemma}
 
 \theoremstyle{definition}
 \newtheorem{defn}[thm]{Definition}
 \theoremstyle{remark}
 
 \newtheorem*{ex}{Example}
 \numberwithin{equation}{section}

\begin{document}

%-------------------------------------------------------------------------
% editorial commands: to be inserted by the editorial office
%
%\firstpage{1} \volume{228} \Copyrightyear{2004} \DOI{003-0001}
%\seriesextraline{This is the Concrete Title of this Book\br H.E. R and S.T.C. W, Eds.}
%\firstpage{1}

\title[]
 {On Nilpotent Multipliers of Pairs of Groups}

%----------Author 1
\author[A. Hokmabadi]{Azam Hokmabadi}

\address{%
Department of Mathematics\\
 Faculty of Sciences\\
 Payame Noor University\\
 19395-4697 Tehran\\
 Iran}

\email{ahokmabadi@pnu.ac.ir}

%\thanks{}
%----------Author 2
\author[F. Mohammadzadeh]{Fahimeh Mohammadzadeh}
\address{Department of Mathematics\\
 Faculty of Sciences\\
 Payame Noor University\\
 19395-4697 Tehran\\
 Iran}
\email{F.mohamadzade@gmail.com}

\author[ B. Mashayekhy]{Behrooz Mashayekhy}
\address{Department of Mathematics\\
Center of Excellence in Analysis on Algebraic Structures\\
Ferdowsi University of Mashhad\\
 1159-91775 Mashhad \\ Iran}
\email{bmashf@um.ac.ir}

\subjclass{ 20F18, 20E34}

\keywords{ Pair of groups; Nilpotent multiplier; Finitely generated abelian groups. }

\date{January 1, 2004}

\begin{abstract}

In this paper, we determine the structure of the nilpotent multipliers of all pairs $(G,N)$ of finitely generated abelian groups where $N$ admits a complement in $G$. Moreover, some inequalities for the nilpotent multipliers of pairs of finite groups and their factor groups are given.

\end{abstract}

%%% ----------------------------------------------------------------------
\maketitle
%%% ----------------------------------------------------------------------
%\tableofcontents

\section{Introduction}
\label{sec1}

Let $G$ be a group with a free presentation $1 \rightarrow R \rightarrow F \rightarrow G \rightarrow 1$. Then the $c$-nilpotent multiplier of $G$ is defined to be
$$ M^{(c)}(G)= \frac{R \cap \gamma_{c+1}(F)}{[R, \ _cF]}.$$

It is easy to see that $ M^{(c)}(G)$ is independent of the choice of the free presentation of $G$. In particular, $ M^{(1)}(G)$ is the well-known notion $M(G)$, the Schur multiplier of $G$, (see \cite{kar}).\

The structure of the Schur multiplier of a finitely generated abelian group is obtained by I. Schur \cite{kar}. In 1997, M.R.R. Moghaddam and the third author \cite{moma} gave an implicit formula for the $c$-nilpotent multiplier of a finite abelian group.\

The theory of the Schur multiplier was extended for pairs of groups by Ellis \cite{ell}, in 1998. By a pair of groups $(G,N)$, we mean a group $G$ with a normal subgroup $N$. The Schur multiplier of a pair $(G,N)$ of groups is a
functorial abelian group $M(G,N)$ whose principal feature is a
natural exact sequence
$$ H_3(G) \stackrel {\eta}{\rightarrow}H_3(\frac{G}{N})
 \rightarrow M(G,N) \rightarrow M(G) \stackrel {\mu}{\rightarrow}
M(\frac{G}{N})\rightarrow \frac{N}{[N,G]}\rightarrow (G)^{ab}
\stackrel {\alpha}{\rightarrow} (\frac{G}{N})^{ab} \rightarrow 0$$
in which  $H_3(G)$ is the third homology of $G$ with integer coefficients.
In particular, if $N=G$, then $M(G,G)$ is the usual Schur multiplier $M(G)$.\\

Let $(G,N)$ be a pair of groups. Ellis \cite{ell} showed that if $N$ admits a complement in $G$, then
\begin{equation}
M(G,N) \cong \ker(\mu: M(G) \rightarrow M(G/N)).\\
\end{equation}

Let $F/R$ be a free presentation of $G$ and $S$ be a subgroup of $F$ with $N \cong S/R$. If $N$ admits a complement in $G$, then (1.1) implies that
$$M(G,N)= \frac{R \cap [S,F]}{[R,F]},$$
(see \cite{msc}). This fact suggests the definition of the $c$-nilpotent multiplier of a pair $(G,N)$ of groups as follows:
$$M^{(c)}(G,N)= \frac{R \cap [S,\ _cF]}{[R,\ _cF]}.$$
In particular, if $G=N$, then $M^{(c)}(G,G) = M^{(c)}(G)$ is the $c$-nilpotent multiplier of $G$.

In this paper, we study the $c$-nilpotent multiplier of a pair of groups. In the next section, we present a formula for the $c$-nilpotent multipliers (and consequently for the Schur multipliers) of all pairs $(G,N)$ of finitely generated abelian groups where $N$ has a complement in $G$.
In the final section, we give some inequalities for the order, the exponent and the minimal number of
generators of the $c$-nilpotent multipliers of pairs of finite groups and their factor groups.

\section{ Nilpotent multipliers of pairs of finitely generated abelian groups }
\label{sec2}

In this section, we intend to find the structure of the $c$-nilpotent multiplier of a pair $(G,N)$ of finitely generated abelian groups, where $N$ has a complement in $G$. The proof relies on basic commutators and their properties.

\begin{defn} (\cite{hall})
Let $X$ be an independent subset of a free group, and select an
arbitrary total order for $X$. We define the basic commutators on
$X$, their weight \textit{wt}, and the ordering among them as
follows:

(1) \ The elements of $X$ are basic commutators of weight one,
ordered according to the total order previously chosen.

(2) \ Having defined the basic commutators of weight less than $n$,
the basic commutators of weight $n$ are the $c_k=[c_i,c_j]$, where

(a) \ $c_i$ and $c_j$ are basic commutators and $wt(c_i)+wt(c_j)=n$,
and

(b) \ $c_i>c_j$, and if $c_i=[c_s,c_t]$, then $c_j\geq c_t$.

(3) \ The basic commutators of weight $n$ follow those of weight
less than $n$. The basic commutators of weight $n$ are ordered among
themselves lexicographically; that is, if $[b_1,a_1]$ and
$[b_2,a_2]$ are basic commutators of weight $n$, then $[b_1,a_1]\leq
[b_2,a_2]$ if and only if $b_1<b_2$, or $b_1=b_2$ and $a_1<a_2$.
\end{defn}

M. Hall \cite{hall} proved that if $F$ is the free group on a finite set $X$, then the basic
commutators of weight $n$ on $X$ provide a basis for the free abelian
group $\gamma_n(F)/\gamma_{n+1}(F)$. The number of these basic commutators is given by Witt formula.

\begin{thm} (The Witt formula \cite{hall}) The number of basic
commutators of weight $n$  on $ d $  generators is given by the
following formula
$$ \chi_{n}(d)= \frac{1}{n}\sum_{m|n}\mu(m)
d^{\frac{n}{m}},$$ where $\mu(m)$ is the M\"{o}bius function, which is defined to be
 $$\mu(m)=\left\{\begin{array}{ll}
 1& ; m=1,\\
 0& ; m=p_1^{\alpha_1}...p_k^{\alpha_k} \ \ \  \exists \alpha_{i} > 1,\\
(-1)^s& ; m=p_1...p_s, \end{array}\right.$$ where the $p_i$'s are distinct prime numbers.
\end{thm}

Hereafter, let $G$ be a finitely generated abelian group with $G=N \oplus K$ where $N= \textbf{Z}^{(l)}\oplus \textbf{Z}_{r_1}\oplus \cdots \oplus \textbf{Z}_{r_t}$ and $K= \textbf{Z}^{(s)}\oplus \textbf{Z}_{r_{t+1}}\oplus \cdots \oplus \textbf{Z}_{r_n}$, such that $r_i | r_{i+1}$, for all $1 \leq i \leq n-1$. Put $m=l+s$, and $G_i\cong \textbf{Z}$, for all $1 \leq i \leq m$, and $G_{m+j} \cong \textbf{Z}_{r_j}$, for all $1\leq j \leq n$. Let
$$1 \rightarrow R_i=1 \rightarrow F_i = \langle y_i \rangle \rightarrow G_i \rightarrow 1 $$
be a free presentation of the infinite cyclic group $G_i$, for all $1 \leq i \leq m$, and let
$$1 \rightarrow R_j=\langle x_j^{r_j} \rangle \rightarrow F_{m+j} = \langle x_j \rangle \rightarrow G_{m+j} \rightarrow 1 $$ be a free presentation of $G_{m+j}$, for all $1 \leq j \leq n$.

Put $Y_1 =  \{y_1,y_2, \ldots , y_l \}$, $Y_2 =  \{y_{l+1},y_{l+2}, \ldots , y_m \}$, $X_1= \{ x_1,\ldots ,x_t \}$, $X_2= \{ x_{t+1},\ldots ,x_n \}$ and $Y=Y_1 \cup Y_2$, $X=X_1 \cup X_2$. Then it is easy to see that $G=N \oplus K$ has the following free presentation
$$ 1 \rightarrow R = T \gamma_{2}(F) \rightarrow F \stackrel{\theta}{\rightarrow} G \rightarrow 1 ,$$
where $F$ is the free group on $X \cup Y$, and $T=\langle x_1^{r_1},\ldots ,x_n^{r_n} \rangle^F $.
Considering the natural map $\theta : F \rightarrow G$,
we have $\theta ^{-1}(N)=SR$ with $S= \langle Y_1 \cup X_1 \rangle^F$ and so $1 \rightarrow R \rightarrow SR \rightarrow N \rightarrow 1$ is a free presentation of $N$, which implies that $$ M^{(c)}(G,N)=\frac{R \cap [RS,\ _cF]}{[R,\ _cF]}=\frac{[RS,\ _cF]}{[R,\ _cF]}.$$
Hence we have
\begin{equation}
M^{(c)}(G,N) \cong \frac{[S,\ _cF][T,\ _cF]\gamma_{c+2}(F)/\gamma_{c+2}(F)}{[T,\ _cF]\gamma_{c+2}(F)/\gamma_{c+2}(F)}.
\end{equation}

To determine the structure of $M^{(c)}(G,N)$, we need suitable bases for the free abelian groups $[S,\ _cF][T,\ _cF] \gamma_{c+2}(F) /\gamma_{c+2}(F)$ and $[T,\ _cF]\gamma_{c+2}(F)/\gamma_{c+2}(F)$. The authors have already obtained a basis for $[T,\ _cF]\gamma_{c+2}(F)/\gamma_{c+2}(F)$ as follows.

\begin{lem}(Lemma 3.2 in \cite{hmm})
Let $C_i$ be the set of all basic commutators of weight $c+1$ on $\{x_i,\cdots ,x_n, y_1,y_2, \cdots , y_m \}$. Then $[T,\ _cF]\gamma_{c+2}(F)/\gamma_{c+2}(F)$ is a free abelian group with a basis $D= \cup_{i=1}^{n}D_i$, where $$D_i = \{ b^{r_i}\gamma_{c+2}(F)|\  b \in C_i \ and \ x_i\ does \ appear\ in\ b \}.$$
\end{lem}

Now we are ready to prove the main result of this section.

\begin{thm}
With the previous notations and assumptions, the following isomorphism holds.
$$M^{(c)}(G,N) \cong \textbf{Z}^{(f_0)} \oplus \textbf{Z}_{r_1}^{(f_1)} \oplus \cdots \oplus \textbf{Z}_{r_t}^{(f_t)}\oplus \textbf{Z}_{r_{t+1}}^{(f_{t+1}-g_{t+1})} \oplus \cdots \oplus \textbf{Z}_{r_n}^{(f_n - g_n)},$$
where $f_0=\chi_{c+1}(m)-\chi_{c+1}(m-l)$,
$f_i = \chi_{c+1}(m+n-i+1)-\chi_{c+1}(m+n-i)$, for $1 \leq i \leq n$, and
$g_i = \chi_{c+1}(m+n-l-i+1)-\chi_{c+1}(m+n-l-i)$, for $t+1 \leq i \leq n$.

\end{thm}

\begin{proof}
In order to determine the structure of $M^{(c)}(G,N)$, we need to find a suitable basis for the free abelian group
$[S,\ _cF][T,\ _cF]\gamma_{c+2}(F)/\gamma_{c+2}(F)$. Let $E$ be the set of all basic commutators of weight $c+1$ on $X \cup Y$ in which at least one of the $ x_{1},\ldots , x_{t}, y_1,\ldots, y_l $ does appear. Put
$$\bar{E}= \{ b\gamma_{c+2}(F) | b\in E \}.$$ Then $ D \cup \bar{E}$ generates the free abelian group
$[S,\ _cF][T,\ _cF]\gamma_{c+2}(F)/\gamma_{c+2}(F)$. Recall that the distinct basic commutators are linearly independent (see \cite{hall}). Hence $\hat{E}=D' \cup \bar{E}$ is a basis for the mentioned free abelian group where $D'$ is the set of all elements $b^{r_i} \gamma_{c+1}(F)$ such that $b$ is a basic commutator of weight $c+1$ on $X_2 \cup Y_2$ in which one of the elements of $X_2$ does appear. In order to determine the structure of the group $M^{(c)}(G,N)$, we present $\hat{E}$ as follows:

$$ \hat{E} = (A_1-A_2) \cup (\cup_{i=1}^t B_i) \cup  (\cup_{i=t+1}^n (B_i - N_i )) \cup  (\cup_{i=t+1}^{n}H_i) $$

where \\
$A_1 = \{ b \gamma_{c+2}(F)|\  b \ is\ a\ basic\ commutator\ of\ weight\ c+1\ on\ Y  \}$,\\
$A_2 = \{ b \gamma_{c+2}(F)|\  b \ is\ a\ basic\ commutator\ of\ weight\ c+1\ on\ Y_2  \}$,\\
$B_i = \{ b \gamma_{c+2}(F)|\  b \ is\ a\ basic\ commutator\ of\ weight\ c+1\ on\ $

$\hspace {1.1 cm} \{x_i,x_{i+1}, \ldots , x_{n} \} \cup  Y\ such\ that\  x_{i}\ does \ appear\ in\ b  \}$, \\
%$M_i = \{ b \gamma_{c+2}(F)|\  b \ is\ a\ basic\ commutator\ of\ weight\ c+1\ on\ \{x_{s_i},x_{s_{i+1}}, \ldots , x_{s_{n-t}} \} \cup Y\ such\ that\  x_{s_i}\ appears\ in\ b  \}$,\\
$N_i = \{ b \gamma_{c+2}(F)|\  b \ is\ a\ basic\ commutator\ of\ weight\ c+1\ on\ $

 $ \hspace {1.1 cm} \{x_{i},x_{i+1}, \ldots , x_{n} \} \cup Y_2 \ such\ that\  x_i \ does\ appear\ in\ b  \}$,\\
$H_i = \{ b^{r_i} \gamma_{c+2}(F)|\  b\gamma_{c+2}(F) \in N_i \}$,\\

On the other hand, Lemma 2.3 provides a basis for the free abelian group $[T,\ _cF]\gamma_{c+2}(F)/\gamma_{c+2}(F)$ as follows.
$$D=(\cup_{i=1}^t B'_i) \cup  (\cup_{i=t+1}^n (B'_i - H_i) ) \cup  (\cup_{i=t+1}^{n}H_i),$$
where $B'_i = \{ b^{r_i} \gamma_{c+2}(F)|\  b\gamma_{c+2}(F) \in B_i \}$, for all $1\leq i\leq n$.
Now considering (2.1) and the obtained bases for the free abelian groups $[T,\ _cF]\gamma_{c+2}(F)/\gamma_{c+2}(F)$ and $[S,\ _cF][T,\ _cF]\gamma_{c+2}(F)/\gamma_{c+2}(F)$, we can conclude that $M^{(c)}(G,N)$ is a finitely generated abelian group in which the number of copies of $\textbf{Z}$ is $|A_1|-|A_2|$ and the number of copies of $\textbf{Z}_{r_i}$ is $|B_i|$, for $1 \leq i \leq t$ and it is $|B_i|-|N_i|$, for $ t+1\leq i \leq n$. On the other hand,\\
$|A_1|=\chi_{c+1}(m)$, $|A_2|=\chi_{c+1}(m-l)$,\\
$|B_i|= \chi_{c+1}(m+n-i+1)-\chi_{c+1}(m+n-i)$, for $1 \leq i \leq n$, \\
$|N_i|= \chi_{c+1}(m+n-l-i+1)-\chi_{c+1}(m+n-l-i)$, for $t+1 \leq i \leq n$.\\
Now putting $f_0=|A_1|-|A_2|$, $f_i=|B_i|$ and $g_i=|N_i|$, for all $1 \leq i \leq n$, we have $$M^{(c)}(G,N) \cong \textbf{Z}^{(f_0)} \oplus \textbf{Z}_{r_1}^{(f_1)} \oplus \cdots \oplus \textbf{Z}_{r_t}^{(f_t)}\oplus \textbf{Z}_{r_{t+1}}^{(f_{t+1}-g_{t+1})} \oplus \cdots \oplus \textbf{Z}_{r_n}^{(f_n - g_n)}.$$
\end{proof}

Note that the extra condition $r_t |r_{t+1}$ in the above theorem is not essential in the process of determining the structure of $M^{(c)}(G,N)$. In fact without this condition the structure of $M^{(c)}(G,N)$ is too complicated to state. Also, the mentioned condition helps us to state the proof of Theorem 2.4 more clear and understandable. The following example shows that the mentioned condition is not essential and the above theorem holds for all pairs $(G,N)$ of finitely generated abelian groups such that $N$ admits a complement in $G$ (without any extra condition).

\begin{ex}

Let $\langle x_1 | x_1^{p^2} \rangle \cong \textbf{Z}_{p^2}$, $\langle x_2 | x_2^{p^4} \rangle \cong \textbf{Z}_{p^4}$, $\langle x_3 | x_3^{p^3} \rangle \cong \textbf{Z}_{p^3}$, $\langle x_4 | x_4^{p^5} \rangle \cong \textbf{Z}_{p^5}$, and $\langle y_i \rangle \cong \textbf{Z}$, for all $1 \leq i \leq m$.
Put $G=N \oplus K$, where $N= \langle y_1 \rangle \oplus \cdots \oplus  \langle y_l \rangle \oplus  \langle x_1 \rangle \oplus  \langle x_2 \rangle$  and $K= \langle y_{l+1} \rangle \oplus \cdots \oplus  \langle y_{m} \rangle \oplus  \langle x_3 \rangle \oplus  \langle x_4 \rangle$.
 Put $X_1=\{ x_1, x_2 \}$, $X_2=\{ x_3,x_4 \}$, $X=X_1 \cup X_2$, and $Y_1=\{ y_1,\ldots,y_l \}$, $Y_2=\{ y_{l+1},\ldots,y_{m} \}$, $Y=Y_1 \cup Y_2$. Then
$ 1 \rightarrow R = T \gamma_{2}(F) \rightarrow F {\rightarrow} G \rightarrow 1 ,$ is a free presentation of $G$
and $1 \rightarrow R \rightarrow SR \rightarrow N \rightarrow 1$ is a free presentation of $N$ where $F$ is the free group on $X \cup Y$, $T=\langle x_1^{p^2}, x_2^{p^4},x_3^{p^3},x_4^{p^5} \rangle^F $ and $S= \langle Y_1 \cup X_1 \rangle^F$.
Define \\
$A_1 = \{ b \gamma_{c+2}(F)|\  b \ is\ a\ basic\ commutator\ of\ weight\ c+1\ on\ Y  \}$,\\
$A_2 = \{ b \gamma_{c+2}(F)|\  b \ is\ a\ basic\ commutator\ of\ weight\ c+1\ on\ Y_2  \}$,\\
$C_1= \{ b \gamma_{c+2}(F)|\  b \ is\ a\ basic\ commutator\ of\ weight\ c+1\ on $

$\hspace {3.4 cm} X \cup Y \ such \ that \ x_1 \ does\ appear \ in \ b \} $,\\
$C_2= \{ b \gamma_{c+2}(F)|\  b \ is\ a\ basic\ commutator\ of\ weight\ c+1\ on$

$\hspace {2.5 cm} \{x_2,x_4\} \cup Y \ such \ that \ x_2 \ does\ appear \ in \ b \} $,\\
$C_3= \{ b \gamma_{c+2}(F)|\  b \ is\ a\ basic\ commutator\ of\ weight\ c+1\ on$

$\hspace {2 cm} \{x_2,x_3,x_4\} \cup Y \ such \ that \ x_3 \ does\ appear \ in \ b  \} $,\\
$C_4= \{ b \gamma_{c+2}(F)|\  b \ is\ a\ basic\ commutator\ of\ weight\ c+1\ on$

$\hspace {3.1 cm} \{x_4\} \cup Y \ such \ that \ x_4 \ does\ appear \ in \ b \} $,\\
$N_3= \{ b \gamma_{c+2}(F)|\  b \ is\ a\ basic\ commutator\ of\ weight\ c+1\ on$

$\hspace {2.5 cm} \{x_3,x_4\} \cup Y_2 \ such \ that \ x_3 \ does\ appear \ in \ b \} $,\\
$N_4= \{ b \gamma_{c+2}(F)|\  b \ is\ a\ basic\ commutator\ of\ weight\ c+1\ on$

$\hspace {3 cm} \{x_4\} \cup Y_2 \ such \ that \ x_4 \ does\ appear \ in \ b \} $,\\
$D_1= \{ b^{p^2} \gamma_{c+2}(F)| \ b \gamma_{c+2}(F) \in C_1 \}$\\
$D_2= \{ b^{p^4} \gamma_{c+2}(F)| \ b \gamma_{c+2}(F) \in C_2 \}$\\
$D_3= \{ b^{p^3} \gamma_{c+2}(F)| \ b \gamma_{c+2}(F) \in C_3 \}$\\
$D_4= \{ b^{p^5} \gamma_{c+2}(F)| \ b \gamma_{c+2}(F) \in C_4 \}$\\
$H_3= \{ b^{p^3} \gamma_{c+2}(F)| \ b \gamma_{c+2}(F) \in N_3 \}$\\
$H_4= \{ b^{p^5} \gamma_{c+2}(F)| \ b \gamma_{c+2}(F) \in N_4 \}$\\

Then using an argument similar to the proof of Theorem 2.4, one can obtain that
$\hat{E}=(A_1-A_2) \cup (C_1 \cup C_2 \cup (C_3-N_3)\cup (C_4-N_4)) \cup (H_3 \cup H_4)$
is a basis for $[S,\ _cF][T,\ _cF]\gamma_{c+2}(F)/\gamma_{c+2}(F)$ and $D= D_1 \cup D_2 \cup (D_3 -H_3) \cup (D_4-H_4) \cup (H_3 \cup H_4)$ is a basis for  $[T,\ _cF]\gamma_{c+2}(F)/\gamma_{c+2}(F)$. Therefore by (2.1) we have
$$M^{(c)}(G,N) \cong \textbf{Z}^{(|A_1-A_2|)} \oplus \textbf{Z}_{p^2}^{(|D_1|)} \oplus \textbf{Z}_{p_3}^{(|D_3-H_3|)}\oplus \textbf{Z}_{p^4}^{(|D_2|)} \oplus \textbf{Z}_{p^5}^{(|D_4-H_4|)},$$
and hence
$$M^{(c)}(G,N) \cong \textbf{Z}^{(f_0)} \oplus \textbf{Z}_{p^2}^{(f_1)} \oplus \textbf{Z}_{p_3}^{(f_2-g_3)}\oplus \textbf{Z}_{p^4}^{(f_3)} \oplus \textbf{Z}_{p^5}^{(f_4-g_4)},$$
where
 $f_0=\chi_{c+1}(m)-\chi_{c+1}(m-l)$,
$f_i = \chi_{c+1}(m+4-i+1)-\chi_{c+1}(m+4-i)$, for $1 \leq i \leq 4$, and
$g_i = \chi_{c+1}(m-l+4-i+1)-\chi_{c+1}(m-l+4-i)$, for $3 \leq i \leq 4$.
\end{ex}

%-------------------------------------------------------------------

% ----------------------------------------------------------------
\section{ Some inequalities for the $c$-nilpotent multiplier of a pair of groups}

In this section, we give some inequalities for the order, the exponent and the minimal number of
generators of the $c$-nilpotent multipliers of pairs of finite groups and their factor groups. We recall the following lemmas that we need them in the sequel.\\

\begin{lem} (Theorem 2.2 in \cite{mss}) Let $(G,N) $ be a pair of finite groups and, $K$ be a normal subgroup of $G$ contained in $N$.
Then \\
(i)  $|\frac{ K \cap [N,\ _cG]}{ [K,\ _cG]}||M^{(c)}(G,N)|=|M^{(c)}(\frac{G}{K},\frac{N}{K})||M^{(c)}(G,K)|$;\\
(ii) $d(M^{(c)}(G,N))\leq d(M^{(c)}(\frac{G}{K},\frac{N}{K}))+d(M^{(c)}(G,K))$;\\
(iii) $exp(M^{(c)}(G,N))\leq exp(M^{(c)}(\frac{G}{K},\frac{N}{K}))exp(M^{(c)}(G,K))$;\\
(iv) $d(M^{(c)}(\frac{G}{K},\frac{N}{K}))\leq d(M^{(c)}(G,N))+d(\frac{ K \cap [N,\ _cG]}{ [K,\ _cG]})$;\\
(v) $exp(M^{(c)}(\frac{G}{K},\frac{N}{K}))$ divides $ exp(M^{(c)}(G,N)) exp(\frac{ K \cap [N,\ _cG]}{ [K,\ _cG]}),$\\
where $d(X)$ is the minimal number of generators of the group $X$.
\end{lem}

\begin{lem} (Lemma 22 in \cite{mmh}) Let $F/R$ be a free presentation of a group $G$  and
$B$ a normal subgroup of $G$, with $B = S/R$. Then there exists the following epimorphism
$$\otimes ^{c+1}(B , G/B) \longrightarrow  \frac{[S,\ _{c}F]}{[R,\ _{c}F]
[S,\ _{c+1}F] \prod_{i=2}^{c+1}\gamma_{c+1}(S,F)_{i}} ,$$ in which
for all $2\leq i\leq c$, $\gamma_{c+1}(S,F)_{i}= [D_1, D_2, \ldots,
D_{c+1}]$ such that $D_{1}=D_{i}=S $ and $D_{j}=F$, for all $j
\neq 1,i$, and   $\otimes^{c+1}(B,G/B)= B\otimes G/B\otimes ... \otimes G/B$ involves $c$
copies of $G/B$.
\end{lem}

\begin{lem} (Lemma 3.1 in \cite{mss}) Let
$H$ and $N$ be subgroups of a group $G$ and $N=N_0 \supseteq N_1
\supseteq \cdots$ a chain of normal subgroups of $N$ such that
$[N_i,G]\subseteq N_{i+1}$, for all $i \in \mathbf{N}$. Then $[N_i,[H,
_jG]] \subseteq N_{i+j+1}$, for all positive integers $i,j$.
\end{lem}

Now we state the first result of this section that is an extension of Corollary 2.3 in \cite{mss}.\\

\begin{thm} Let $(G,N) $ be a pair of
finite groups and $K$ be a central subgroup of $G$ contained in $N$. Let $ F/R$ be a free presentation of $G$ and $T$ be a normal subgroup of the free group $F$
such that  $K = T/R$. Then \\
(i) $ \frac{| K \cap [N , _cG]|}{|[K, _cG]|}|M^{(c)}(G,N)|\mid
|M^{(c)}(\frac{G}{K},\frac{N}{K})||\otimes
^{c+1}(K,\frac{G}{K})||\frac{[R,\ _{c}F]
\prod_{i=2}^{c+1}\gamma_{c+1}(T,F)_{i}}{[R,\ _{c}F]}|$;\\
(ii) $d(M^{(c)}(G,N))\leq
d(M^{(c)}(\frac{G}{K},\frac{N}{K}))+d(\otimes
^{c+1}(K,\frac{G}{K}))+d(\frac{[R,\ _{c}F]
\prod_{i=2}^{c+1}\gamma_{c+1}(T,F)_{i}}{[R,\ _{c}F]})$;\\
(iii) $exp(M^{(c)}(G,N))\mid
exp(M^{(c)}(\frac{G}{K},\frac{N}{K}))exp(\otimes
^{c+1}(K,\frac{G}{K}))exp(\frac{[R,\ _{c}F]
\prod_{i=2}^{c+1}\gamma_{c+1}(T,F)_{i}}{[R,\ _{c}F]}).$
\end{thm}

\begin{proof}
(i) Since $K$ is a central subgroup of $G$, we have $[T,F]\leq R$. Then Lemma 3.2 implies
the epimorphism $$\otimes ^{c+1}(K , \frac{G}{K})
\longrightarrow  \frac{[T,\ _{c}F]}{[R,\ _{c}F]
\prod_{i=2}^{c+1}\gamma_{c+1}(T,F)_{i}}.$$
On the other hand, we have
$$|\frac{(R \cap [T, _cF])/[R, _cF]}{([R,\ _{c}F] \prod_{i=2}^{c+1}\gamma_{c+1}(T,F)_{i})/[R,\ _{c}F]} |
=| \frac{[T, _cF]}{[R,\ _{c}F] \prod_{i=2}^{c+1}\gamma_{c+1}(T,F)_{i}}|$$
Therefore $|\frac{(R \cap [T, _cF])/[R, _cF]}{([R,\ _{c}F] \prod_{i=2}^{c+1}\gamma_{c+1}(T,F)_{i})/[R,\ _{c}F]} |$
divides $|\otimes^{c+1}(K,\frac{G}{K})|$
and so $$|M^{(c)}(G,K)| \mid |\otimes^{c+1}(K,\frac{G}{K})||\frac{[R,\ _{c}F] \prod_{i=2}^{c+1}\gamma_{c+1}(T,F)_{i}}{[R,\ _{c}F]}|.$$ Hence the result holds by Lemma 3.1. The proof of (ii) and (iii) are similar.
\end{proof}

The following theorem generalizes Theorem 3.2 in \cite{mss} and also Theorem C in \cite{mmh}.\\

\begin{thm} Let $(G,N) $ be a pair of finite nilpotent groups of class $t$. Then \\
(i)\

(a) If $t \geq c+1$, then\\
$  |[N, _{t-1}G]||M^{(c)}(G,N)|$ divides
$|M^{(c)}(\frac{G}{[N, _{t-1}G]},\frac{N}{[N, _{t-1}G]})||\otimes^{c+1}([N, _{t-1}G], \frac{G}{Z_{t-1}(N,G)})|;$\

 (b) If $t < c+1$, then\\
$  |[N, _{c}G]||M^{(c)}(G,N)|$ divides $ |M^{(c)}(\frac{G}{[N,
_{t-1}G]},\frac{N}{[N, _{t-1}G]})||\otimes
^{c+1}([N, _{t-1}G], \frac{G}{Z_{t-1}(N,G)})|$;\\
(ii) $ d(M^{(c)}(G,N))\leq d(M^{(c)}(\frac{G}{[N,
_{t-1}G]},\frac{N}{[N, _{t-1}G]}))+d(\otimes
^{c+1}([N, _{t-1}G], \frac{G}{Z_{t-1}(N,G)}))$;\\
(iii) $ exp(M^{(c)}(G,N))\leq exp(M^{(c)}(\frac{G}{[N,
_{t-1}G]},\frac{N}{[N, _{t-1}G]})) exp(\otimes
^{c+1}([N, _{t-1}G], \frac{G}{Z_{t-1}(N,G)})).$
\end{thm}

\begin{proof}
Let $G\cong F/R$ be a free presentation of $G$. Let $N\cong S/R$ and $Z_{t-i}(N,G)\cong T_i /R $, for all $0 \leq i \leq t $. Consider the following chain
$$ S=T_{0}\supseteq ...\supseteq T_k \supseteq \cdots \supseteq T_{t-1}\supseteq T_{t}=R\supseteq [R,F]\supseteq \cdots\supseteq [R, _cF].$$
Since $[T_k,F]\subseteq T_{k+1}$, we have
$[T_{i}, [S, _{t-1}F]]\subseteq [R, _iF]$ by Lemma 3.3.
This inclusion induces the following epimorphism.
$$ \hspace{0.7 cm} \otimes ^{c+1}(\frac{[S, _{t-1}F]R}{R},\frac{F}{T_{t-1}}) \rightarrow
\frac{[[S, _{t-1}F]R, _cF]}{[R, _cF]}$$ $$ s[R, _cF] \otimes x_1T_{t-1}\otimes...\otimes x_cT_{t-1}\mapsto [s,x_1,...,x_c][R, _cF]$$
So we have
\begin{equation}
|\frac{[[S, _{t-1}F]R, _cF]}{[R, _cF]}| \mid  |\otimes ^{c+1}(\frac{[S, _{t-1}F]R}{R},\frac{F}{T_{t-1}})|
\end{equation}
On the other hand, considering $K=[N, _{t-1}G]$ in Lemma 3.1, we have
if $t \geq c+1$, then
$$|[N, _{t-1}G]||M^{(c)}(G,N)|=|M^{(c)}(\frac{G}{[N, _{t-1}G]},\frac{N}{[N, _{t-1}G]})||\frac{[[S, _{t-1}F]R, _cF]}{[R, _cF]}|,$$
and if $t<c+1  $, then
$$|[N, _{c}G]||M^{(c)}(G,N)|=|M^{(c)}(\frac{G}{[N,
_{t-1}G]},\frac{N}{[N, _{t-1}G]})||\frac{[[S, _{t-1}F]R,
_cF]}{[R, _cF]}|.$$

Now (i) follows by (3.1). One can obtain (ii) and (iii) similarly.
\end{proof}

The next result gives another upper bound for the order, the exponent and the minimal number of generators of the $c$-nilpotent multiplier of a pair of finite groups. This theorem is an extension of Theorem 3.4 in \cite{mss} and also generalizes Theorem B in \cite{mmh}

\begin{thm} Let $(G,N) $ be a pair of
finite nilpotent groups of class at most $t$ ($t \geq 2$). Then \\
(i) $|[N, _cG]||M^{(c)}(G,N)|$ divides $|M^{(c)}(\frac{G}{[N,
G]},\frac{N}{[N,G]})|\prod ^{t-1} _{i=1}|\otimes ^{c+1}([N,
_iG],\frac{G}{[N, _{i}G]}|$;\\
(ii) $d(M^{(c)}(G,N)) \leq
d(M^{(c)}(\frac{G}{[N, G]},\frac{N}{[N,G]}))+\sum ^{t-1}
_{i=1}d(\otimes ^{c+1}([N, _iG],\frac{G}{[N, _{i}G]})$;\\
(iii) $exp(M^{(c)}(G,N))$ divides $ exp(M^{(c)}(\frac{G}{[N, G]},\frac{N}{[N,G]}))\prod ^{t-1}
_{i=1}exp(\otimes ^{c+1}([N, _iG],\frac{G}{[N, _{i}G]})$.
\end{thm}

\begin{proof}
(i) Let $F$, $S$ and $R$ be as in Theorem 3.5. Considering
$K=[N,G]$ in Lemma 3.1, we have
$$| [N, _cG]||M^{(c)}(G,N)|=
|M^{(c)}(\frac{G}{[N,G]},\frac{N}{[N,G]})||M^{(c)}(G,[N,G])|[N, _{c+1}G]|.$$
On the other hand,
\begin{eqnarray*}
 |[N, _{c+1}G]||M^{(c)}(G,[N,G])|&=&|\frac{[S,
_{c+1}F]R}{R}||\frac{(R \cap [S,F, _cF])[R, _cF]}{[R,
_cF]}|\\&=&|\frac{[[S, F]R, _{c}F]}{[R, _cF]}|\\&=&|\frac{[[S,
_{t} F]R, _{c}F]}{[R, _cF]}|\prod ^{t-1} _{i=1}|\frac{[[S,
_{i} F]R, _{c}F]}{[[S, _{i+1} F]R, _{c}F]}|.
 \end{eqnarray*}
 By the assumption of theorem,
$1=[N, _{t}G]=([S,\ _{t} F]R)/R$ and hence
 $[[S,\ _{t} F]R,\ _{c}F]=[R,\ _cF]$. Therefore
 $$|[N,\ _cG]||M^{(c)}(G,N)|=|M^{(c)}(\frac{G}{[N, G]},\frac{N}{[N,G]})
 |\prod ^{t-1}_{i=1}|\frac{[[S,\ _{i} F]R, _{c}F]}{[[S,\ _{i+1} F]R,\ _{c}F]}|.$$
On the other hand, for all $1 \leq i \leq t-1$ ,
$$\prod ^{c+1} _{j=2} \gamma_{c+1}([[S, _iF]R,F])_j
[[S,\ _{i} F]R, \ _{c+1}F]\leq [[S,\ _{i+1}F]R,\ _cF].$$ Considering this inequality, Lemma 3.2 implies that $$|\frac{[[S,\ _{i} F]R,
_{c}F]}{[[S,\ _{i+1} F]R,\ _{c}F]}| \mid |\otimes ^{c+1}([N, _iG],\frac{G}{[N, _{i}G])}|,$$ and hence the assertion follows.\\
(ii) Let $r(G)$ be the special rank of $G$. Using 3.1 and 3.2, we obtain
\begin{eqnarray*}
d(M^{(c)}(G,N)) &\leq & d(M^{(c)}(\frac{G}{[N,G]},\frac{N}{[N,G]}))+d(M^{(c)}(G,[N,G]))\\& \leq&
d(M^{(c)}(\frac{G}{[N, G]},\frac{N}{[N,G]}))+r(\frac{(R \cap [S,F, _cF])[R, _cF]}{[R, _cF]})\\&\leq&
d(M^{(c)}(\frac{G}{[N, G]},\frac{N}{[N,G]}))+r(\frac{[[S,F]R,_{c}F]}{[R, _{c}F]})\\&\leq&
d(M^{(c)}(\frac{G}{[N,G]},\frac{N}{[N,G]}))+\sum^{t-1} _{i=1}r(\frac{[[S, _{i} F]R, _{c}F]}{[[S, _{i+1} F]R, _{c}F]})\\&\leq&
d(M^{(c)}(\frac{G}{[N, G]},\frac{N}{[N,G]}))+\sum ^{t-1}_{i=1}d(\otimes ^{c+1}([N, _iG],\frac{G}{[N, _{i}G]})).\\
 \end{eqnarray*}
The last inequality holds because $r(A)=d(A)$ for any finite abelian group $A$.\\
(iii) It can be proved easily by an argument similar to the proof of (i).
\end{proof}

%\section*{Acknowledgement }

% ------------------------------------------------------------------------
\end{document}